\documentclass{article}
\usepackage{tikz}
\usetikzlibrary{positioning}
\tikzset{main node/.style={circle,fill=blue!20,draw,minimum size=0.4cm,inner sep=0pt},}
\usepackage{amsmath}
\usepackage{amsthm}
\usepackage{amsfonts}
\usepackage{listings}
\usepackage{float}
\usepackage{hyperref}
\usepackage[labelformat=simple]{subcaption}

\usepackage[numbered]{matlab-prettifier}
\usepackage[T1]{fontenc}
\usepackage{enumitem}
\restylefloat{table}
\newtheorem{theorem}{Theorem}
\newtheorem{lemma}{Lemma}
\newtheorem{corollary}[theorem]{Corollary}
\theoremstyle{definition}
\newtheorem{definition}{Definition}
\newtheorem{remark}{Remark}

\usepackage{thmtools}
\usepackage{thm-restate}
\usepackage{cleveref}

\usepackage[new]{old-arrows}

\DeclareMathOperator{\chara}{char}

\DeclareMathOperator{\Tor}{Tor}
\DeclareMathOperator{\vsupp}{vsupp}

\lstdefinestyle{mymatstyle}{%
	style=Matlab-editor,
	basicstyle=\mlttfamily,
	frame=leftline,
	numberstyle=\scriptsize,
	xleftmargin=1.8em,
}

\newcommand{\Addresses}{{
  \bigskip
  \footnotesize

  Peter Phelan, \textsc{National University of Ireland, Galway}\par\nopagebreak
  \textit{Email address}: \texttt{p.phelan3@nuigalway.ie}
  
}}
  
\begin{document}

\title{Parity Binomial Edge Ideals with Pure Resolutions}
\author{Peter Phelan}
\date{National University of Ireland, Galway \\
    \today}
\maketitle
\begin{abstract}
We provide a characterisation of all graphs whose parity binomial edge ideals have pure resolutions. In particular, we show that the minimal free resolution of a parity binomial edge ideal is pure if and only if the corresponding graph is a complete bipartite graph, or a disjoint union of paths and odd cycles.
\end{abstract}
\section{Introduction}
The parity binomial edge ideal of a simple undirected graph $G$ was introduced in \cite{pbei} by Kahle, Sarmiento and Windisch. It's here we learn that these ideals do not have a square-free Gr{\"o}bner basis, and that they are radical if and only if $G$ is bipartite or the ground field has $\chara(\mathbb{K}) \neq 2$.
\begin{definition} The parity binomial edge ideal of a graph $G$ is given by
$$ J_{G} = (x_ix_j-y_iy_j \, | \, \{i,j\} \in E(G) ) \subseteq \mathbb{K}[x_i,y_i \, | \, i \in V(G)]$$
\end{definition}
Kumar later revealed in \cite{intersection} that $J_{G}$ is a complete intersection if and only if $G$ is a disjoint union of paths and odd cycles, followed by a characterisation of all graphs for which $J_{G}$ is an almost complete intersection. For a complete graph $G$, it was shown by Badiane, Burke and Sk{\"o}ldberg in \cite{basis} that the universal Gr{\"o}bner basis and the Graver basis of $J_{G}$ coincide; in addition, the Hilbert-Poincaré series of $J_{G}$ was also determined in \cite{hilb} by Do Trong Hoang and Kahle. This short survey reveals a growing interest in these ideals, aided by their appearance in various other fields; for example, in statistics, biology and machine learning. \\ \\
In this paper, we are interested in studying parity binomial edge ideals with pure resolutions. These resolutions have been the subject of serious interest in recent years, often described as the "building blocks" of Betti diagrams. In fact, conjectures posed by Boij and S{\"o}derberg, later proven by Eisenbud and Schreyer, demonstrate how pure resolutions can be used to examine the range of admissible Betti diagrams. An excellent review of this topic can be found in \cite{boij}.
\begin{definition}
Suppose $R$ is a polynomial ring and $\mathbb{F}$ is the graded free resolution of a finitely generated graded $R$-module $M$ with components $F_i = \bigoplus_{j} R(-j)^{\beta_{i,j}}$. The resolution $\mathbb{F}$ is called pure if each $F_{i}$ is generated in a single degree, so $\mathbb{F}$ can be written as follows:
$$ \dots \longrightarrow R(-d_{3})^{\beta_{3,d_{3}}} \longrightarrow R(-d_{2})^{\beta_{2,d_{2}}} \longrightarrow R(-d_{1})^{\beta_{1,d_{1}}} \longrightarrow M \longrightarrow 0 $$
\end{definition}
Another study of note is the work of Dariush Kiani and Sara Saeedi Madani in \cite{pure}, where they prove that a binomial edge ideal has a pure resolution if and only if $G$ is a complete graph, a complete bipartite graph or a disjoint union of paths. A natural question arising from this result is whether a similar characterisation exists for parity binomial edge ideals. This question serves as the basis of our investigation, and brings us to the main result of the paper.
\begin{restatable}{thm}{pure}
\label{thm:pure}
Let $G$ be a simple connected graph and let $J_{G}$ be the parity binomial edge ideal of $G$. If the minimal free resolution of $J_{G}$ is pure, then $G$ is a bipartite graph or an odd cycle.
\end{restatable}
To prove this theorem, we consider the contrapositive statement; suppose $G$ is a simple connected graph that is non-bipartite and not an odd cycle, we claim that $J_{G}$ cannot have a pure resolution in this case. The main technique we use here is to consider an induced subgraph $H$ of $G$, for which there are two pairs of integers $(i,j)$ and $(i,k)$, $j \neq k$, such that $\beta_{i,j}(R/J_{H})$ and $\beta_{i,k}(R/J_{H})$ are nonzero. By applying \cite[Proposition 8]{retract_2} to these Betti numbers, we deduce that $\beta_{i,j}(R/J_{G})$ and $\beta_{i,k}(R/J_{G})$ must also be nonzero, and thus, $J_{G}$ must have a nonpure resolution. This means the proof of Theorem \ref{thm:pure} can be reduced to the study of the graded Betti numbers of induced subgraphs of $G$. Although this proposition deals with the binomial edge ideal of a pair of graphs, an identical proof also applies to $J_{G}$. We restate this proposition as a remark.
\begin{remark}
\label{remark:rmk1}
Let $G$ be a graph and $H$ an induced subgraph of $G$. Then for all nonnegative integers $i,j$, we have the following inequality of graded Betti numbers: 
$$ \beta_{i,j}(R/J_{H}) \leq \beta_{i,j}(R/J_{G}) $$
\end{remark}
To determine whether some Betti number $\beta_{i,j}(R/J_{H})$ of an induced subgraph of $G$ is nonzero, we will use a method of short exact sequences. In particular, we consider a short exact sequence containing $R/J_{H}$ and two other components for which $\beta_{i,j}$ is known. Then, we consider the long exact $\Tor$ sequence induced by this short exact sequence, and apply the rank-nullity theorem to deduce an equality of Betti numbers. This method relies on a couple more remarks, the next of which is found in \cite[Remark 1.2]{pbei}.
\begin{remark}
\label{remark:rmk2}
Suppose $G$ is a bipartite graph and $V_1,V_2$ are its disjoint sets of vertices. Consider the ring automorphism on $\mathbb{K}[\Vec{x},\Vec{y}]$ which exchanges $x_i$ and $y_i$ for all $i \in V_1$ and leaves the remaining variables invariant. This automorphism maps the binomial edge ideal of $G$ to $J_{G}$, so that any statement proven for the binomial edge ideal of a bipartite graph $G$ must also hold for $J_{G}$, and vice versa.
\end{remark}
Our final remark is a restatement of \cite[Theorem 5.3]{bipartite}, which will be used in conjunction with Remark \ref{remark:rmk2} to study the parity binomial edge ideals of complete bipartite graphs.
\begin{remark}
\label{remark:rmk3}
The binomial edge ideal of the complete bipartite graph $K_{m,n}$ has its Betti diagram in the following form:
\begin{center}
\begin{tabular}{ c|ccccc } 
  & 0 & 1 & 2 & \dots & p \\
 \hline
 0 & 1 & 0 & 0 & \dots & 0 \\ 
 1 & 0 & $mn$ & 0 & \dots & 0 \\ 
 2 & 0 & 0 & $\beta_{2,4}$ & \dots & $\beta_{p,p+2}$
\end{tabular}
\end{center}
where
  \[
    p=\left\{
            \begin{array}{lll}
              m & \text{if} & n=1;\\
              2m+n-2 & \text{if} & m \geq n>1
            \end{array}              
        \right.
  \]
\end{remark}
In Section 2, we begin with a series of lemmas required for the proof of our main result. In particular, we consider a collection of Betti numbers that appear in this proof, and investigate whether these are zero or nonzero. This is followed by the proof of Theorem \ref{thm:pure}. Finally, we show in Section 3 that the minimal free resolution of $J_{G}$ is pure if and only if $G$ is a complete bipartite graph, or a disjoint union of paths and odd cycles. This result will be formulated as a corollary to Theorem \ref{thm:pure}, and relies on the characterisation of binomial edge ideals provided by Kiani and Madani.
\newpage
\section{Parity Binomial Edge Ideals with Pure Resolutions}
Let $C_n$ denote the cycle of length $n$, let $P_n$ denote the path graph on $n$ vertices and let $K_{n,m}$ denote the complete bipartite graph on $n$ and $m$ vertices.
\begin{lemma}
\label{lemma:lem1}
The minimal free resolution of $J_{K_{1,3}}$ is a pure resolution.
\end{lemma}
\proof
We know from Remark \ref{remark:rmk3} that the binomial edge ideal of a complete bipartite graph has a pure resolution. Since $K_{1,3}$ is a complete bipartite graph, it follows that $J_{K_{1,3}}$ must have a pure resolution by Remark \ref{remark:rmk2}. \hfill \qedsymbol
\begin{lemma}
\label{lemma:lem2}
$\beta_{3,6}(R/J_{C_{3}})$ and $\beta_{3,6}(R/J_{P_{4}})$ are nonzero.
\end{lemma}
\proof
Since $C_3$ is an odd cycle and $P_4$ is a path, we know that $J_{C_{3}}$ and $J_{P_{4}}$ are complete intersection ideals and therefore resolved by the Koszul complex. Recall that the Koszul complex on a sequence of $k$ elements satisfies $\beta_{i,j} = {k\choose i}$ whenever $j=2i$, and $\beta_{i,j} = 0$ otherwise. Since each of these ideals have more than two generators, and thus, the sequence associated to the Koszul complex has more than two elements, we conclude that $\beta_{3,6}(R/J_{C_{3}}) > 0$ and $\beta_{3,6}(R/J_{P_{4}}) > 0$ as claimed. \hfill \qedsymbol
\begin{lemma}
\label{lemma:lem3}
$\beta_{3,5}(R/J_{K_{2,2}})$ and $\beta_{3,5}(R/J_{K_{1,3}})$ are nonzero, while $\beta_{3,5}(R/J_{C_{3}})$, $\beta_{3,5}(R/J_{P_{4}})$, $\beta_{3,5}(R/J_{P_{3}})$, $\beta_{3,5}(R/J_{P_{2}})$, $\beta_{2,5}(R/J_{K_{1,3}})$, $\beta_{2,5}(R/J_{C_3})$, $\beta_{2,5}(R/J_{P_3})$, $\beta_{2,5}(R/J_{P_2})$ are all equal to zero.
\end{lemma}
\proof
We already know from Lemma \ref{lemma:lem2} that $\beta_{3,5}(R/J_{C_{3}}) = 0$, $\beta_{3,5}(R/J_{P_{4}}) = 0$, $\beta_{3,5}(R/J_{P_{3}}) = 0$, $\beta_{3,5}(R/J_{P_{2}}) = 0$, $\beta_{2,5}(R/J_{C_{3}}) = 0$, $\beta_{2,5}(R/J_{P_{3}}) = 0$ and $\beta_{2,5}(R/J_{P_{2}}) = 0$ as each of these ideals are resolved by the Koszul complex. Then, since $K_{1,3}$ and $K_{2,2}$ are complete bipartite graphs, Remarks \ref{remark:rmk2} \& \ref{remark:rmk3} imply that $\beta_{2,5}(R/J_{K_{1,3}}) = 0$, $\beta_{3,5}(R/J_{K_{1,3}}) > 0$ and $\beta_{3,5}(R/J_{K_{2,2}}) > 0$ as claimed. \hfill \qedsymbol \\ \\
For the next result, we will need the notion of a multi-homogeneous polynomial.
\begin{definition}
A polynomial $f$ in $\mathbb{K}[x_{1,1},x_{1,2} \dots x_{1,k_{1}},x_{2,1},x_{2,2} \dots x_{n,k_{n}}]$ is called multi-homogeneous of multi-degree $(d_1 \dots d_n)$ if it is homogeneous of degree $d_i$ in variables $x_{i,1} \dots x_{i,k_{i}}$ for all nonnegative integers $i \leq n$.
\end{definition}
It is known that parity binomial edge ideals are homogeneous with respect to the $\mathbb{N}^{n}$ grading on $R$, given by $\deg(x_i) = \deg(y_i) = \Vec{e_i}$ for all $i \in [n]$, where $\Vec{e_{i}}$ the $i$th standard basis vector in $\mathbb{N}^{n}$. This grading allows us to consider parity binomial edge ideals as multi-homogeneous of multi-degree $(d_1 \dots d_n)$, where each $d_i$ corresponds to the variables $x_i$ and $y_i$ for all $i$. Therefore, any intersection of parity binomial edge ideals must also be multi-homogeneous with respect to this grading.
\newpage
\begin{lemma}
\label{lemma:lem4}
$J_{C_3} \cap J_{K_{1,3}}$, $J_{P_3} \cap J_{K_{1,3}}$ and $J_{P_2} \cap J_{K_{1,3}}$ are generated in degree  at least four.
\end{lemma}
\proof
The ideals in this lemma refer particularly to those in Figures \ref{fig:4}, \ref{fig:5} and \ref{fig:6}, which are written as shown: $J_{C_{3}} = (x_1x_3-y_1y_3, x_1x_4-y_1y_4, x_3x_4-y_3y_4)$, $J_{K_{1,3}} = (x_1x_2-y_1y_2, x_2x_3-y_2y_3, x_2x_4-y_2y_4)$, $J_{P_{3}} = (x_1x_4-y_1y_4, x_3x_4-y_3y_4)$ and $J_{P_{2}} = (x_1x_3-y_1y_3)$. Suppose $\Vec{\alpha} \in \mathbb{N}^{2n}$ and $x^{\Vec{\alpha}}$ is a monomial in $R$, we define the support of $x^{\Vec{\alpha}}$ by the set $\vsupp(x^{\Vec{\alpha}}) = \{\, i \; | \; x_i \; \text{or} \; y_i \;\text{divides} \; x^{\Bar{\alpha}} \, \}$. If we let $f$ be a multi-homogeneous, nonzero generator of $J=J_{C_{3}} \cap J_{K_{1,3}}$, then we know that $f$ must have degree at least three and that $|\vsupp(f)|$ must also be at least three, as the graphs $C_3$ and $K_{1,3}$ have no edges in common. \\ \\
Let us then assume that the degree of $f$ is equal to three. Since $f$ is multi-homogeneous, we know that $|\vsupp(f)|$ must also be equal to three. The graph $K_{1,3}$ requires that $2 \in \vsupp(f)$. So we may assume, without loss of generality, that $\vsupp(f) = \{1,2,3\}$. Then the vector space $V$ of degree three monomials supported by $\vsupp(f)$ has a basis $\{x_1x_2x_3, x_1x_2y_3, x_1y_2x_3, y_1x_2x_3, x_1y_2y_3, y_1x_2y_3, \newline y_1y_2x_3, y_1y_2y_3\}$, and $f$ is an element of $V \cap J$. \\ \\
The vector space $V_{C_3}$ given by $J_{C_{3}} \cap V$ has a basis $\{x_1x_2x_3-y_1x_2y_3, x_1y_2x_3-y_1y_2y_3\}$, and the vector space $V_{K_{1,3}}$ given by $J_{K_{1,3}} \cap V$ has a basis $\{x_1x_2x_3-y_1y_2x_3, x_1x_2y_3-y_1y_2y_3, x_1x_2x_3-x_1y_2y_3, y_1x_2x_3-y_1y_2y_3\}$. Note that each basis element contained in $V_{C_3} + V_{K_{1,3}}$ has a monomial term not shared by any other basis element, so these basis elements are linearly independent. From this, it is clear that $\dim(V_{C_{3}} + V_{K_{1,3}}) = \dim(V_{C_3}) + \dim(V_{K_{1,3}})$; therefore $\dim(V_{C_3} \cap V_{K_{1,3}}) = 0$ and $V_{C_3} \cap V_{K_{1,3}} = V \cap J = 0$, a contradiction. We conclude that $f$ must have degree at least four, and thus, $J_{C_3} \cap J_{K_{1,3}}$ is generated in degree at least four. \\ \\
To deal with the remaining intersections, we first note that $K_{1,3}$ and $P_3$ require $2,4 \in \vsupp(f)$ for $f$ a multi-homogeneous, nonzero generator of $J_{P_3} \cap J_{K_{1,3}}$. So we may assume, without loss of generality, that $\vsupp(f) = \{1,2,4\}$. This is simply a relabelling of the previous case, as we have a triangle with two edges from one graph, and one edge from the second graph. Thus, we apply the exact same argument used above and conclude $J_{P_3} \cap J_{K_{1,3}}$ is generated in degree at least four. Similarly, we note that $K_{1,3}$ and $P_{2}$ require $\vsupp(f) = \{1,2,3\}$, so the above argument implies $J_{P_2} \cap J_{K_{1,3}}$ must also be generated in degree at least four and the claim follows. \hfill \qedsymbol
\begin{lemma}
\label{lemma:lem5}
$\beta_{3,5}(R/(J_{C_3} \cap J_{K_{1,3}}))$, $\beta_{3,5}(R/(J_{P_3} \cap J_{K_{1,3}}))$ and \\ $\beta_{3,5}(R/(J_{P_2} \cap J_{K_{1,3}})) $ are all equal to zero.
\end{lemma}
\proof
We know from Lemma \ref{lemma:lem4} that these ideals are generated in degree at least four, so the relations on the generators must be of degree at least five, and the relations on those must be of degree at least six. The claim follows. \hfill \qedsymbol \\ \\
We now proceed with the proof of Theorem \ref{thm:pure}.
\newpage
\pure*
\begin{proof} [Proof of Theorem~\ref{thm:pure}]
We prove this statement by contraposition. Suppose $G$ is a simple connected graph that is non-bipartite and not an odd cycle, we claim that the minimal free resolution of the parity binomial edge ideal $J_{G}$ of $G$ cannot be a pure resolution. This follows as one can always find an induced subgraph $H$ of $G$, such that the minimal free resolution of $J_{H}$ is nonpure, which implies the resolution of $J_{G}$ must also be nonpure by Remark \ref{remark:rmk1} above. We proceed by detailing how these subgraphs are determined. \\ \\
Recall that a non-bipartite graph must contain an odd cycle. If there is more than one odd cycle in $G$, we restrict our attention to one of minimal length. Since $G$ is a connected graph, this odd cycle must share an edge with some vertex $v$ not contained in the cycle; in other words, it must have at least one vertex with degree at least $3$. Now we let $H$ be the induced subgraph of $G$ given by this vertex, its two closest neighbours in the odd cycle, and the vertex $v$. \\ \\ This leaves us with four cases of induced subgraphs (up to permutation of indices), namely, Figures \ref{fig:1a}, \ref{fig:1b}, \ref{fig:1c} and \ref{fig:2a}, which we denote $G_1,G_2,G_3$ and $G_4$ respectively.
\begin{figure}[H]
$ \quad $
\begin{subfigure}[b]{0.25\textwidth}
\begin{tikzpicture}
\node[main node] (2) {$2$};
\node[main node] (1) [below left = 0.6cm and 1cm of 2]  {$1$};
\node[main node] (3) [below right = 0.6cm and 1cm of 2] {$3$};
\node[main node] (v) [above = 1.2cm of 2] {$v$};

\path[draw,thick]
(1) edge node {} (2)
(2) edge node {} (3)
(2) edge node {} (v);
\end{tikzpicture}
\caption{$G_1$}
\label{fig:1a}
\end{subfigure}
$ \qquad $
\begin{subfigure}[b]{0.25\textwidth}
\begin{tikzpicture}
\node[main node] (2) {$2$};
\node[main node] (1) [below left = 0.6cm and 1cm of 2]  {$1$};
\node[main node] (3) [below right = 0.6cm and 1cm of 2] {$3$};
\node[main node] (v) [above = 1.2cm of 2] {$v$};

\path[draw,thick]
(1) edge node {} (2)
(2) edge node {} (3)
(2) edge node {} (v)
(1) edge node {} (v)
(1) edge node {} (3)
(3) edge node {} (v);
\end{tikzpicture}
\caption{$G_2$}
\label{fig:1b}
\end{subfigure}
$ \qquad $
\begin{subfigure}[b]{0.25\textwidth}
\begin{tikzpicture}
\node[main node] (2) {$2$};
\node[main node] (1) [below left = 0.6cm and 1cm of 2]  {$1$};
\node[main node] (3) [below right = 0.6cm and 1cm of 2] {$3$};
\node[main node] (v) [above = 1.2cm of 2] {$v$};

\path[draw,thick]
(1) edge node {} (2)
(2) edge node {} (3)
(2) edge node {} (v)
(1) edge node {} (v)
(3) edge node {} (v);
\end{tikzpicture}
\caption{$G_3$}
\label{fig:1c}
\end{subfigure}
\caption{}
\end{figure}
\begin{figure}[H]
$ \qquad \qquad \qquad $
\begin{subfigure}[b]{0.25\textwidth}
\begin{tikzpicture}
\node[main node] (2) {$2$};
\node[main node] (1) [below left = 0.6cm and 1cm of 2]  {$1$};
\node[main node] (3) [below right = 0.6cm and 1cm of 2] {$3$};
\node[main node] (v) [above = 1.2cm of 2] {$v$};

\path[draw,thick]
(1) edge node {} (2)
(2) edge node {} (3)
(2) edge node {} (v)
(1) edge node {} (3);
\end{tikzpicture}
\caption{$G_4$}
\label{fig:2a}
\end{subfigure}
$ \qquad \qquad $
\begin{subfigure}[b]{0.25\textwidth}
\begin{tikzpicture}
\node[main node] (2) {$2$};
\node[main node] (1) [left = 1cm of 2]  {$1$};
\node[main node] (3) [right = 1cm of 2] {$3$};
\node[main node] (4) [right = 1cm of 3] {$4$};
\node[main node] (v) [below = 1cm of 2] {$v$};

\path[draw,thick]
(1) edge node {} (2)
(2) edge node {} (3)
(3) edge node {} (4)
(2) edge node {} (v);
\end{tikzpicture}
\caption{$G_5$}
\label{fig:2b}
\end{subfigure}
\caption{}
\end{figure}
\begin{figure}[H]
$ \qquad \qquad \,$
\begin{subfigure}[b]{0.4\textwidth}
\begin{tikzpicture}
\node[main node] (2) {$2$};
\node[main node] (1) [left = 1cm of 2]  {$1$};
\node[main node] (3) [right = 1cm of 2] {$3$};
\node[main node] (4) [below = 1cm of 3] {$4$};
\node[main node] (v) [below = 1cm of 2] {$v$};

\path[draw,thick]
(1) edge node {} (2)
(2) edge node {} (3)
(3) edge node {} (4)
(2) edge node {} (v)
(4) edge node {} (v);
\end{tikzpicture}
\caption{$G_6$}
\label{fig:3a}
\end{subfigure}
$ \qquad $
\begin{subfigure}[b]{0.4\textwidth}
\begin{tikzpicture}
\node[main node] (2) {$2$};
\node[main node] (1) [left = 1cm of 2]  {$1$};
\node[main node] (3) [right = 1cm of 2] {$3$};
\node[main node] (4) [below = 1cm of 3] {$4$};
\node[main node] (v) [below = 1cm of 2] {$v$};

\path[draw,thick]
(1) edge node {} (2)
(2) edge node {} (3)
(3) edge node {} (4)
(2) edge node {} (v)
(4) edge node {} (v)
(1) edge node {} (4);
\end{tikzpicture}
\caption{$G_7$}
\label{fig:3b}
\end{subfigure}
\caption{}
\end{figure}
We already know from Lemma \ref{lemma:lem1} that $J_{G_1}$ has a pure resolution. For the moment, we focus our attention on the remaining ideals $J_{G_2}$, $J_{G_3}$ and $J_{G_4}$, where we show they have nonpure resolutions. First, we note that $G_2$, $G_3$ and $G_4$ have $C_3$ as an induced subgraph. Therefore, by applying Remark \ref{remark:rmk1} \& Lemma \ref{lemma:lem2}, we deduce that $\beta_{3,6}(J_{G_{2}}), \beta_{3,6}(J_{G_{3}}) $ and $ \beta_{3,6}(J_{G_{4}}) $ must all be nonzero. \\ \\
We now claim $\beta_{3,5}$ is nonzero for each of these ideals, which we prove with the aid of some short exact sequences. Beginning with the graph $G_2$, we note $C_3$ and $K_{1,3}$ can be considered as subgraphs of $G_2$ such that neither share an edge, as shown below.
\begin{figure}[H]
$ \qquad $
\begin{subfigure}[b]{0.25\textwidth}
\begin{tikzpicture}
\node[main node] (2) {$2$};
\node[main node] (1) [below left = 0.6cm and 1cm of 2]  {$1$};
\node[main node] (3) [below right = 0.6cm and 1cm of 2] {$3$};
\node[main node] (4) [above = 1.2cm of 2] {$4$};

\path[draw,thick]
(1) edge node {} (2)
(2) edge node {} (3)
(2) edge node {} (4)
(1) edge node {} (4)
(1) edge node {} (3)
(3) edge node {} (4);
\end{tikzpicture}
\caption{$G_2$}
\label{fig:4a}
\end{subfigure}
$ \qquad $
\begin{subfigure}[b]{0.25\textwidth}
\begin{tikzpicture}
\node[main node] (2) {$2$};
\node[main node] (1) [below left = 0.6cm and 1cm of 2]  {$1$};
\node[main node] (3) [below right = 0.6cm and 1cm of 2] {$3$};
\node[main node] (4) [above = 1.2cm of 2] {$4$};

\path[draw,thick]
(1) edge node {} (3)
(1) edge node {} (4)
(3) edge node {} (4);
\end{tikzpicture}
\caption{$C_3$}
\label{fig:4b}
\end{subfigure}
$ \qquad $
\begin{subfigure}[b]{0.25\textwidth}
\begin{tikzpicture}
\node[main node] (2) {$2$};
\node[main node] (1) [below left = 0.6cm and 1cm of 2]  {$1$};
\node[main node] (3) [below right = 0.6cm and 1cm of 2] {$3$};
\node[main node] (4) [above = 1.2cm of 2] {$4$};

\path[draw,thick]
(1) edge node {} (2)
(2) edge node {} (3)
(2) edge node {} (4);
\end{tikzpicture}
\caption{$K_{1,3}$}
\label{fig:4c}
\end{subfigure}
\caption{}
\label{fig:4}
\end{figure}
Consider the following short exact sequence, and the long exact Tor sequence it induces:
$$ 0 \longrightarrow R/(J_{C_3} \cap J_{K_{1,3}}) \longrightarrow R/J_{C_3} \oplus R/J_{K_{1,3}} \longrightarrow R/(J_{C_3}+J_{K_{1,3}}) \longrightarrow 0$$
\begin{align*}
\dots & \longrightarrow \Tor_{3}(R/(J_{C_3} \cap J_{K_{1,3}}),\mathbb{K})_5 \longrightarrow \Tor_{3}(R/J_{C_3} \oplus R/J_{K_{1,3}},\mathbb{K})_5 \longrightarrow \\ 
& \longrightarrow \Tor_{3}(R/(J_{C_3}+J_{K_{1,3}}),\mathbb{K})_5  \longrightarrow \Tor_{2}(R/(J_{C_3} \cap J_{K_{1,3}}),\mathbb{K})_5 \longrightarrow \\ 
& \longrightarrow \Tor_{2}(R/J_{C_3} \oplus R/J_{K_{1,3}},\mathbb{K})_5 \longrightarrow \dots
\end{align*}
We know that $\Tor_{3}(R/(J_{C_3} \cap J_{K_{1,3}}),\mathbb{K})_5 = 0$ and $\Tor_{2}(R/J_{C_3} \oplus R/J_{K_{1,3}},\mathbb{K})_5 = 0$ by Lemmas \ref{lemma:lem3} \& \ref{lemma:lem5}. This leads to the short exact sequence of vector spaces:
\begin{align*}
0 & \longrightarrow \Tor_{3}(R/J_{C_3} \oplus R/J_{K_{1,3}},\mathbb{K})_5  \longrightarrow \Tor_{3}(R/(J_{C_3}+J_{K_{1,3}}),\mathbb{K})_5 \longrightarrow \\ 
& \longrightarrow \Tor_{2}(R/(J_{C_3} \cap J_{K_{1,3}}),\mathbb{K})_5 \longrightarrow 0
\end{align*}
It is easily shown that $J_{C_3}+J_{K_{1,3}} = J_{G_{2}}$, and therefore: $$ \beta_{3,5}(R/J_{G_2}) = \beta_{3,5}(R/J_{C_3}) + \beta_{3,5}(R/J_{K_{1,3}}) + \beta_{2,5}(R/(J_{C_3} \cap J_{K_{1,3}})) $$
By applying Lemma \ref{lemma:lem3} to this equality, we deduce that $ \beta_{3,5}(R/J_{G_2}) > 0 $ as claimed. We now proceed with the case $G_3$, where $K_{1,3}$ and $P_{3}$ are considered as subgraphs of $G_3$ with no overlapping edges as shown.
\begin{figure}[H]
$ \qquad $
\begin{subfigure}[b]{0.25\textwidth}
\begin{tikzpicture}
\node[main node] (2) {$2$};
\node[main node] (1) [below left = 0.6cm and 1cm of 2]  {$1$};
\node[main node] (3) [below right = 0.6cm and 1cm of 2] {$3$};
\node[main node] (4) [above = 1.2cm of 2] {$4$};

\path[draw,thick]
(1) edge node {} (2)
(2) edge node {} (3)
(2) edge node {} (4)
(1) edge node {} (4)
(3) edge node {} (4);
\end{tikzpicture}
\caption{$G_3$}
\label{fig:5a}
\end{subfigure}
$ \qquad $
\begin{subfigure}[b]{0.25\textwidth}
\begin{tikzpicture}
\node[main node] (2) {$2$};
\node[main node] (1) [below left = 0.6cm and 1cm of 2]  {$1$};
\node[main node] (3) [below right = 0.6cm and 1cm of 2] {$3$};
\node[main node] (4) [above = 1.2cm of 2] {$4$};

\path[draw,thick]
(1) edge node {} (2)
(2) edge node {} (3)
(2) edge node {} (4);
\end{tikzpicture}
\caption{$K_{1,3}$}
\label{fig:5b}
\end{subfigure}
$ \qquad $
\begin{subfigure}[b]{0.25\textwidth}
\begin{tikzpicture}
\node[main node] (2) {$2$};
\node[main node] (1) [below left = 0.6cm and 1cm of 2]  {$1$};
\node[main node] (3) [below right = 0.6cm and 1cm of 2] {$3$};
\node[main node] (4) [above = 1.2cm of 2] {$4$};

\path[draw,thick]
(1) edge node {} (4)
(3) edge node {} (4);
\end{tikzpicture}
\caption{$P_{3}$}
\label{fig:5c}
\end{subfigure}
\caption{}
\label{fig:5}
\end{figure}
Consider the following short exact sequence, and the long exact Tor sequence it induces:
$$ 0 \longrightarrow R/(J_{P_3} \cap J_{K_{1,3}}) \longrightarrow R/J_{P_3} \oplus R/J_{K_{1,3}} \longrightarrow R/(J_{P_3}+J_{K_{1,3}}) \longrightarrow 0$$
\begin{align*}
\dots & \longrightarrow \Tor_{3}(R/(J_{P_3} \cap J_{K_{1,3}}),\mathbb{K})_5 \longrightarrow \Tor_{3}(R/J_{P_3} \oplus R/J_{K_{1,3}},\mathbb{K})_5 \longrightarrow \\ 
& \longrightarrow \Tor_{3}(R/(J_{P_3}+J_{K_{1,3}}),\mathbb{K})_5  \longrightarrow \Tor_{2}(R/(J_{P_3} \cap J_{K_{1,3}}),\mathbb{K})_5 \longrightarrow \\ 
& \longrightarrow \Tor_{2}(R/J_{P_3} \oplus R/J_{K_{1,3}},\mathbb{K})_5 \longrightarrow \dots
\end{align*}
We know that $\Tor_{3}(R/(J_{P_3} \cap J_{K_{1,3}}),\mathbb{K})_5 = 0$ and $\Tor_{2}(R/J_{P_3} \oplus R/J_{K_{1,3}},\mathbb{K})_5 = 0$ by Lemmas \ref{lemma:lem3} \& \ref{lemma:lem5}. This leads to the short exact sequence of vector spaces:
\begin{align*}
0 & \longrightarrow \Tor_{3}(R/J_{P_3} \oplus R/J_{K_{1,3}},\mathbb{K})_5  \longrightarrow \Tor_{3}(R/(J_{P_3}+J_{K_{1,3}}),\mathbb{K})_5 \longrightarrow \\ 
& \longrightarrow \Tor_{2}(R/(J_{P_3} \cap J_{K_{1,3}}),\mathbb{K})_5 \longrightarrow 0
\end{align*}
It is easily shown that $J_{P_3}+J_{K_{1,3}} = J_{G_{3}}$, and therefore: $$ \beta_{3,5}(R/J_{G_3}) = \beta_{3,5}(R/J_{P_3}) + \beta_{3,5}(R/J_{K_{1,3}}) + \beta_{2,5}(R/(J_{P_3} \cap J_{K_{1,3}})) $$
Again, by applying Lemma \ref{lemma:lem3} to this equality, we find that $ \beta_{3,5}(R/J_{G_3}) > 0$ as claimed. Finally, we look at the case $G_4$, where $K_{1,3}$ and $P_{2}$ are considered as subgraphs of $G_4$ with no overlapping edges as shown.
\begin{figure}[H]
$ \qquad $
\begin{subfigure}[b]{0.25\textwidth}
\begin{tikzpicture}
\node[main node] (2) {$2$};
\node[main node] (1) [below left = 0.6cm and 1cm of 2]  {$1$};
\node[main node] (3) [below right = 0.6cm and 1cm of 2] {$3$};
\node[main node] (4) [above = 1.2cm of 2] {$4$};

\path[draw,thick]
(1) edge node {} (2)
(2) edge node {} (3)
(2) edge node {} (4)
(1) edge node {} (3);
\end{tikzpicture}
\caption{$G_4$}
\label{fig:6a}
\end{subfigure}
$ \qquad  $
\begin{subfigure}[b]{0.25\textwidth}
\begin{tikzpicture}
\node[main node] (2) {$2$};
\node[main node] (1) [below left = 0.6cm and 1cm of 2]  {$1$};
\node[main node] (3) [below right = 0.6cm and 1cm of 2] {$3$};
\node[main node] (4) [above = 1.2cm of 2] {$4$};

\path[draw,thick]
(1) edge node {} (2)
(2) edge node {} (3)
(2) edge node {} (4);
\end{tikzpicture}
\caption{$K_{1,3}$}
\label{fig:6b}
\end{subfigure}
$ \qquad $
\begin{subfigure}[b]{0.25\textwidth}
\begin{tikzpicture}
\node[main node] (2) {$2$};
\node[main node] (1) [below left = 0.6cm and 1cm of 2]  {$1$};
\node[main node] (3) [below right = 0.6cm and 1cm of 2] {$3$};
\node[main node] (4) [above = 1.2cm of 2] {$4$};

\path[draw,thick]
(1) edge node {} (3);
\end{tikzpicture}
\caption{$P_{2}$}
\label{fig:6c}
\end{subfigure}
\caption{}
\label{fig:6}
\end{figure}
Consider the following short exact sequence, and the long exact Tor sequence it induces:
$$ 0 \longrightarrow R/(J_{P_2} \cap J_{K_{1,3}}) \longrightarrow R/J_{P_2} \oplus R/J_{K_{1,3}} \longrightarrow R/(J_{P_2}+J_{K_{1,3}}) \longrightarrow 0$$
\begin{align*}
\dots & \longrightarrow \Tor_{3}(R/(J_{P_2} \cap J_{K_{1,3}}),\mathbb{K})_5 \longrightarrow \Tor_{3}(R/J_{P_2} \oplus R/J_{K_{1,3}},\mathbb{K})_5 \longrightarrow \\ 
& \longrightarrow \Tor_{3}(R/(J_{P_2}+J_{K_{1,3}}),\mathbb{K})_5  \longrightarrow \Tor_{2}(R/(J_{P_2} \cap J_{K_{1,3}}),\mathbb{K})_5 \longrightarrow \\ 
& \longrightarrow \Tor_{2}(R/J_{P_2} \oplus R/J_{K_{1,3}},\mathbb{K})_5 \longrightarrow \dots
\end{align*}
We know that $\Tor_{3}(R/(J_{P_2} \cap J_{K_{1,3}}),\mathbb{K})_5 = 0$ and $\Tor_{2}(R/J_{P_2} \oplus R/J_{K_{1,3}},\mathbb{K})_5 = 0$ by Lemmas \ref{lemma:lem3} \& \ref{lemma:lem5}. This leads to the short exact sequence of vector spaces:
\begin{align*}
0 & \longrightarrow \Tor_{3}(R/J_{P_2} \oplus R/J_{K_{1,3}},\mathbb{K})_5  \longrightarrow \Tor_{3}(R/(J_{P_2}+J_{K_{1,3}}),\mathbb{K})_5 \longrightarrow \\ 
& \longrightarrow \Tor_{2}(R/(J_{P_2} \cap J_{K_{1,3}}),\mathbb{K})_5 \longrightarrow 0
\end{align*}
It is easily shown that $J_{P_2}+J_{K_{1,3}} = J_{G_{4}}$, and therefore: $$ \beta_{3,5}(R/J_{G_4}) = \beta_{3,5}(R/J_{P_2}) + \beta_{3,5}(R/J_{K_{1,3}}) + \beta_{2,5}(R/(J_{P_2} \cap J_{K_{1,3}})) $$
Once more, by applying Lemma \ref{lemma:lem3} to the equality, we obtain $ \beta_{3,5}(R/J_{G_4}) > 0$ as claimed. Therefore, both $\beta_{3,5}$ and $\beta_{3,6}$ are nonzero for each of the above ideals, and we conclude that $J_{G_2},J_{G_3},J_{G_4}$ have nonpure resolutions. \\ \\
We have seen that if our induced subgraph $H$ is any of $G_2,G_3$ or $G_4$, then we are done, so we need only deal with the case $H=G_1$. Let us begin by introducing an additional vertex to our induced subgraph. This vertex is chosen from the odd cycle such that it neighbours either vertex $1$ or $3$ in Figure \ref{fig:1a} above. We know this vertex exists as our chosen cycle has length at least five. This will lead to a further three cases (up to permutation of indices), namely Figures \ref{fig:2b}, \ref{fig:3a} and \ref{fig:3b} which we will denote $G_5,G_6$ and $G_7$ respectively.
\newpage
Beginning with Figure \ref{fig:2b}, it is clear that both $K_{1,3}$ and $P_{4}$ are induced subgraphs of $G_{5}$, thus $\beta_{3,5}(R/J_{G_5}) > 0 $ and $\beta_{3,6}(R/J_{G_5}) > 0 $ by Lemmas \ref{lemma:lem2} \& \ref{lemma:lem3}. Therefore $J_{G_5}$ has a nonpure resolution. Next, we look at Figure \ref{fig:3a}. It is clear that both $K_{2,2}$ and $P_4$ are induced subgraphs of $G_{6}$, thus $\beta_{3,5}(R/J_{G_6}) > 0 $ and $\beta_{3,6}(R/J_{G_6}) > 0 $ by Lemmas \ref{lemma:lem2} \& \ref{lemma:lem3}. Therefore $J_{G_6}$ has a nonpure resolution.
\\ \\
Finally, we claim Figure \ref{fig:3b} amounts to a contradiction, so $H$ cannot be equal to $G_7$. First, note that the vertices $1$ and $4$ are contained in the odd cycle, so the edge between them is in fact a chord on the cycle. It is known that a chord through an odd cycle divides it into an even cycle and another odd cycle of smaller length. This contradicts the minimality of our chosen cycle. Therefore, we can rule out this final case and the theorem is proven.
\end{proof}
\newpage
\section{Corollaries}
To prove the corollaries below, we must first introduce some results from \cite{pure}. The first of these is found within the proof of \cite[Theorem 2.2]{pure}, while the second is a restatement of \cite[Lemma 2.5]{pure}.
\begin{remark}
\label{remark:cor1}
If the binomial edge ideal of a graph $G$ has a pure resolution, and $G$ is not a complete graph, then $G$ must be a bipartite graph or a disjoint union of paths.
\end{remark}
\begin{remark}
\label{remark:cor2}
Let $I \subset R = \mathbb{K}[x_1 \dots x_n]$ and $J \subset T=\mathbb{K}[x_{n+1} \dots x_m]$ be two graded ideals with pure resolutions.
$$ 0 \longrightarrow R(-d_{p})^{\beta_{p}} \longrightarrow \dots  \longrightarrow R(-d_{1})^{\beta_{1}} \longrightarrow R \longrightarrow R/I \longrightarrow 0 $$
$$ 0 \longrightarrow T(-e_{q})^{\gamma_{q}} \longrightarrow \dots  \longrightarrow T(-e_{1})^{\gamma_{1}} \longrightarrow T \longrightarrow T/J \longrightarrow 0 $$
and let $S=\mathbb{K}[x_1 \dots x_m]$. Then $IS + JS$ has a pure resolution if and only if $e_1 = d_1$, $d_i = id_1$ and $e_j = je_1$ for all $i,j$.
This lemma was proven using the fact that the minimal graded free resolution of $S/(IS+JS)$ is the tensor product of those of $R/I$ and $T/J$, which yields the formula for all $i,j$:
$$ \beta_{i,j}(S/(IS+JS)) = \sum_{t+t'=i,k+k'=j} \beta_{t,k}(R/I) \beta_{t',k'}(T/J) $$
If $G$ is a disjoint union of connected components, then the minimal free resolution of $R/J_{G}$ is equal to the tensor product of the minimal free resolutions arising from those connected components. Therefore, the formula above allows us to compute the Betti numbers $\beta_{i,j}(R/J_{G})$ using the Betti numbers of the connected components.
\end{remark}
\begin{corollary}
\label{remark:cor3}
Let $G$ be a simple connected graph, then $J_{G}$ has a pure resolution if and only if $G$ is an odd cycle, a complete bipartite graph, or a path graph.
\end{corollary}
\begin{proof}
If $G$ is a complete bipartite graph, Remarks \ref{remark:rmk2} \& \ref{remark:rmk3} imply that $J_{G}$ has a pure resolution. If $G$ is an odd cycle or a path graph, then we know from \cite{intersection} that $J_{G}$ is a complete intersection ideal and therefore has a pure resolution. Now we consider the other direction. \\ \\
We have shown in Theorem \ref{thm:pure} that if $J_{G}$ has a pure resolution, then $G$ must be a bipartite graph or an odd cycle. Suppose $G$ is a bipartite graph, then it cannot be a complete graph ($K_{2}$ is also a path graph, so we can ignore it here). Remarks \ref{remark:rmk2} \& \ref{remark:cor1} then imply that $G$ must be a complete bipartite graph or a path graph, which completes the proof.
\end{proof}
\newpage
\begin{corollary}
Let $G$ be a simple graph with no isolated vertices, then $J_{G}$ has a pure resolution if and only if $G$ is a complete bipartite graph, or a disjoint union of paths and odd cycles.
\end{corollary}
\begin{proof}
If $G$ is a complete bipartite graph, then $J_{G}$ has a pure resolution by Remarks \ref{remark:rmk2} \& \ref{remark:rmk3}. If $G$ is a disjoint union of paths and odd cycles, then the parity binomial edge ideal of each connected component is a complete intersection ideal. It follows from Remark \ref{remark:cor2} that $J_{G}$ must also have a pure resolution. \\ \\
Now we consider the other direction and suppose that $J_{G}$ has a pure resolution. The parity binomial edge ideal of each connected component in $G$ must also have a pure resolution by Remark \ref{remark:rmk1}, so Corollary \ref{remark:cor3} implies that each of these components must either be an odd cycle, a complete bipartite graph, or a path graph. Suppose one of these connected components is a complete bipartite graph $K_{m,n}$ where we exclude the path graphs $K_{1,1}$ and $K_{1,2}$. This ensures that $\beta_{2,4}(R/J_{K_{m,n}})$ and $\beta_{3,5}(R/J_{K_{m,n}})$ are nonzero by Remarks \ref{remark:rmk2} \& \ref{remark:rmk3}. \\ \\
Now suppose that another connected component is the path graph $P_{2}$. We know that $\beta_{1,2}(R/J_{P_{2}})$ is nonzero since $J_{P_{2}}$ is a complete intersection ideal, and we also know that $\beta_{3,6}(R/J_{G}) = \beta_{1,2}(R/J_{P_{2}}) \cdot  \beta_{2,4}(R/J_{K_{m,n}})$ by Remark \ref{remark:cor2}. We've just shown that both these Betti numbers are nonzero, which implies $\beta_{3,6}(R/J_{G})$ is nonzero. Since $\beta_{3,5}(R/J_{K_{m,n}})$ is nonzero, Remark \ref{remark:rmk1} then implies that $\beta_{3,5}(R/J_G)$ is nonzero, and therefore $J_{G}$ has a nonpure resolution, a contradiction. This means if $K_{m,n}$ is a connected component of $G$, it must be the only connected component, as any other component will contain $P_{2}$ as an induced subgraph. We conclude that $G$ must either be a complete bipartite graph, or a disjoint union of paths and odd cycles as claimed.
\end{proof}
\subsection*{Acknowledgements}
I would like to thank my supervisor - Emil Sk{\"o}ldberg - for his support and guidance throughout this paper. His help made all this possible.

\Addresses
\end{document}